\documentclass[12pt]{amsart}
\usepackage{a4wide}
\usepackage{graphicx,eepic, bm, times}
\usepackage{amsthm,amsmath,amsfonts,amssymb}
\usepackage{multirow}

\newtheorem{lemma}{Lemma}[section]

\newtheorem{prop}[lemma]{Proposition}
\newtheorem{thm}[lemma]{Theorem}
\newtheorem{cor}[lemma]{Corollary}
\theoremstyle{definition}

\theoremstyle{remark}

\numberwithin{equation}{section} \numberwithin{table}{section}

\begin{document}

\title[Codings for self-similar sets with positive Lebesgue measure]{On the cardinality and complexity of the set of codings for self-similar sets with positive Lebesgue measure}
\author{Simon Baker}
\address{Utrecht University, P.O Box 80125, 3508 TC Utrecht, The Netherlands. E-mail:
simonbaker412@gmail.com}

\date{\today}

\begin{abstract}
Let $\lambda_{1},\ldots,\lambda_{n}$ be real numbers in $(0,1)$ and $p_{1},\ldots,p_{n}$ be points in $\mathbb{R}^{d}$. Consider the collection of maps $f_{j}:\mathbb{R}^{d}\to\mathbb{R}^{d} $ given by $$f_{j}(x)=\lambda_{j} x +(1-\lambda_{j})p_{j}.$$ It is a well known result that there exists a unique compact set $\Lambda\subset \mathbb{R}^{d}$ satisfying $\Lambda=\cup_{j=1}^{n} f_{j}(\Lambda).$ Each $x\in \Lambda$ has at least one coding, that is a sequence $(\epsilon_{i})_{i=1}^{\infty}\in \{1,\ldots,n\}^{\mathbb{N}}$ that satisfies $\lim_{N\to\infty}f_{\epsilon_{1}}\cdots f_{\epsilon_{N}} (0)=x.$ 

We study the size and complexity of the set of codings of a generic $x\in \Lambda$ when $\Lambda$ has positive Lebesgue measure. In particular, we show that under certain natural conditions almost every $x\in\Lambda$ has a continuum of codings. We also show that almost every $x\in\Lambda$ has a universal coding. 

Our work makes no assumptions on the existence of holes in $\Lambda$ and improves upon existing results when it is assumed $\Lambda$ contains no holes. 
\end{abstract}

\maketitle

\section{Introduction}
Let $\lambda\in(\frac{1}{2},1)$ and $I_{\lambda}:=[0,\frac{\lambda}{1-\lambda}].$ Each $x\in I_{\lambda}$ admits a sequence $(\epsilon_{i})_{i=1}^{\infty}\in\{0,1\}^{\mathbb{N}}$ such that $$x=\sum_{i=1}^{\infty}\epsilon_{i}\lambda^{i}.$$  Such a sequence is called a \textit{$\lambda$-expansion} for $x.$ Expansions of this form were pioneered in the papers of R\'{e}nyi \cite{Renyi} and Parry \cite{Parry}. We can study $\lambda$-expansions via the iterated function system defined by the maps $f_{0}(x)=\lambda x$ and $f_{1}(x)=\lambda x+ \lambda.$ It is a straightforward exercise to show that $$f_{\epsilon_{1}}\cdots f_{\epsilon_{N}}(0)=\sum_{i=1}^{N}\epsilon_{i}\lambda^{i}.$$ Therefore $\lim_{N\to \infty}f_{\epsilon_{1}}\cdots f_{\epsilon_{N}}(0)=x$ if and only if $(\epsilon_{i})_{i=1}^{\infty}$ is a $\lambda$-expansion for $x.$

In \cite{Erdos} it was shown that if $\lambda\in(\frac{\sqrt{5}-1}{2},1)$ then every $x\in(0,\frac{\lambda}{1-\lambda})$  has a continuum of $\lambda$-expansions. The endpoints of $I_{\lambda}$ trivially have a unique expansion. In \cite{DaKa} the value $\frac{\sqrt{5}-1}{2}$ was shown to be sharp in the following sense: If $\lambda\in(\frac{1}{2},\frac{\sqrt{5}-1}{2})$ then there exists $x\in(0,\frac{\lambda}{1-\lambda})$ with a unique $\lambda$-expansion. The size of the set of points with unique $\lambda$-expansion was studied further in \cite{GlenSid}, amongst other things it was shown that the set of $x\in I_{\lambda}$ with unique $\lambda$-expansion has positive Hausdorff dimension when $\lambda\in(\frac{1}{2},\lambda^{*})$. Here $\lambda^{*}\approx 0.559$ is the reciprocal of the Komornik Loreti constant introduced in \cite{KomLor}. However, in \cite{Sid} it was shown that Lebesgue almost every $x\in I_{\lambda}$ has a continuum of $\lambda$-expansions for any $\lambda\in(\frac{1}{2},1)$. This almost every result was later generalised to a class of IFS's in \cite{Sidorov4}. We now give details of their generalisation.

Let $\lambda_{1},\ldots,\lambda_{n}$ be real numbers in $(0,1)$ and $p_{1},\ldots,p_{n}$ be points in $\mathbb{R}^{d}$. Consider the collection of maps $f_{j}:\mathbb{R}^{d}\to\mathbb{R}^{d}$ given by 
\begin{equation}
\label{scaling equation}
f_{j}(x)=\lambda_{j} x +(1-\lambda_{j})p_{j}.
\end{equation} There exists a unique compact set $\Lambda\subset \mathbb{R}^{d}$ that satisfies $\Lambda=\cup_{j=1}^{n} f_{j}(\Lambda).$ We refer to $\Lambda$ as the \textit{attractor} for the collection of maps $\{f_{j}\}_{j=1}^{n},$ or when the collection of maps is obvious just the attractor.  Each $x\in \Lambda$ admits a sequence $(\epsilon_{i})_{i=1}^{\infty}\in \{1,\ldots,n\}^{\mathbb{N}}$ such that $\lim_{N\to\infty}f_{\epsilon_{1}}\cdots f_{\epsilon_{N}} (0)=x.$  We refer to such a sequence as a \textit{coding for $x$}. Moreover, the set of $x$ which have a coding is precisely $\Lambda.$ When $\lambda_{1}=\cdots =\lambda_{n}$ we will say that we are in the \textit{homogeneous case}. When there exists $\lambda_{i},\lambda_{j}$ such that $\lambda_{i}\neq \lambda_{j}$ we will say that we are in the \textit{inhomogeneous case}. When we are in the homogeneous case we will denote the common scaling ratio by $\lambda.$

Let $\Omega$ denote the convex hull of $\{p_{1},\ldots,p_{n}\}.$ Without loss of generality we may assume that the dimension of $\Omega$ is $d.$ In \cite{Sidorov4} the author considers the homogeneous case where $\Lambda=\Omega,$ i.e., the case when the attractor has no holes. In particular they show that the property $\Lambda=\Omega$ holds for all $\lambda\geq\frac{d}{d+1}.$ They also proved the following result.

\begin{thm}
\label{Nik's theorem}
Assume $\Lambda=\Omega$ and that we are in the homogeneous case. If there exists $1\leq k<l\leq n$ such that a vertex of $f_{k}(\Omega)$ belongs to the interior of $f_{l}(\Omega)$ then Lebesgue almost every $x\in \Lambda$ has a continuum of codings, and the exceptional set has Hausdorff dimension strictly less than $d$. 
\end{thm}
 
The purpose of this paper is to generalise and strengthen Theorem \ref{Nik's theorem}. Our approach does not make any assumptions on the existence of holes in $\Lambda$ and extends to the inhomogeneous case.

Let $\Lambda$ be as above. We will be interested in the case when $\mathcal{L}(\Lambda)>0$. Here $\mathcal{L}(\cdot)$ denotes the $d$-dimensional Lebesgue measure. Clearly when $\Lambda=\Omega$ then $\mathcal{L}(\Lambda)>0$. However, there are cases when $\Lambda\neq \Omega,$ i.e., the case when our attractor has holes, yet the Lebesgue measure of $\Lambda$ is still positive. Typically, determining whether the attractor of a given IFS has positive Lebesgue measure is a difficult problem.

In \cite{JorPol} the authors consider the case when there are $n$ homogeneous contractions $f_{j}:\mathbb{R}^{2}\to \mathbb{R}^{2}$ of the form $$f_{j}(x)=\lambda x +(c_{j}^{1},c_{j}^{2}),$$  where $(c_{j}^{1},c_{j}^{2})\in\{(a,b)\in\mathbb{Z}^{2}:0\leq a,b\leq k-1\}.$ It is assumed that $n>k.$ If the points $(c_{j}^{1},c_{j}^{2})$ are fixed and $\lambda$ is allowed to vary, the geometry of the associated $\Lambda$ also varies. In particular, if $\lambda$ is sufficiently small then the open set condition is satisfied and the Hausdorff dimension is easy to compute. However, for $\lambda$ sufficiently large the open set condition is not satisfied and determining the dimension of $\Lambda$ is less straightforward. The authors show that for each family of contractions there exists an interval $I\subset (0,1)$ for which $\mathcal{L}(\Lambda)>0$ for almost every $\lambda\in I.$ Moreover, this $I$ is calculated explicitly. Their results imply the existence of a broad class of $\Lambda$ for which $\mathcal{L}(\Lambda)>0$ and $\Lambda$ contains holes.

In \cite{BrMoSid} the case where $p_{1},p_{2}$ and $p_{3}$ are the vertices of an equilateral triangle is studied. For $\lambda=\frac{1}{2}$ our $\Lambda$ is the well known Sierpi\'{n}ski gasket. It can be shown that $\Lambda=\Omega$ if and only if $\lambda \geq \frac{2}{3}$. The authors show that for all $\lambda\geq \lambda^{*}\approx 0.647 $ the associated $\Lambda$ has nonempty interior and therefore has positive Lebesgue measure. Here $\lambda^{*}$ is the appropriate root of $2x^3-2x^2+2x=1.$ It is a consequence of the aforementioned results of \cite{JorPol} that for almost every $\lambda>0.585\ldots$ the associated $\Lambda$ has positive Lebesgue measure.

Our main result is the following generalisation of Theorem \ref{Nik's theorem}.

\begin{thm}
\label{Uncountable thm}
Assume $\sum_{j=1}^{n}\lambda_{j}^{d}\neq 1$ and that $\mathcal{L}(\Lambda)>0$. Then almost every $x\in \Lambda$ has a continuum of codings. 
\end{thm}
When $f_{k}(\Lambda)\cap f_{l}(\Lambda)$ has nonempty interior for some $1\leq k<l\leq n$ we can make a stronger statement.

\begin{thm}
\label{Dimension thm}
Assume $f_{k}(\Lambda)\cap f_{l}(\Lambda)$ has nonempty interior for some $1\leq k<l\leq n.$ Then the set of points that do not have a continuum of codings has Hausdorff dimension strictly less that $d.$
\end{thm}

The expression $\sum_{j=1}^{n}\lambda_{j}^{d}$ occuring in the statement of Theorem \ref{Uncountable thm} appears naturally in the study of IFS's. If $\sum_{j=1}^{n}\lambda_{j}^{d}< 1$ then it is a simple exercise to show $\mathcal{L}(\Lambda)=0.$ Therefore it is only possible for $\mathcal{L}(\Lambda)>0$ when $\sum_{j=1}^{n}\lambda_{j}^{d}\geq 1.$ The condition $\sum_{j=1}^{n}\lambda_{j}^{d}\neq 1 $ stated in Theorem \ref{Uncountable thm} is not a technical condition and is in fact essential. It will be shown in Corollary \ref{unique cor} that if $\sum_{j=1}^{n}\lambda_{j}^{d}=1$ then almost every $x\in\Lambda$ has a unique coding. It is natural to ask whether there exists self-similar sets with positive Lebesgue measure when $\sum_{j=1}^{n}\lambda_{j}^{d}=1$. However, it is straightforward to construct examples when this equation is satisfied. For example, consider the case where $p_{1}=0,$ $p_{2}=1$ and $\lambda_{1}=\lambda_{2}.$ In this case $\sum_{j=1}^{n}\lambda_{j}^{d}=1$ when $\lambda=\frac{1}{2}$. The associated $\Lambda$ is the interval $[0,1],$ which clearly has positive Lebesgue measure.

In this paper, as well as studying the cardinality of the set of codings of a generic $x\in \Lambda$, we also study the complexity of these codings. In the context of $\lambda$-expansions we say that a $\lambda$-expansion of $x,$ the sequence $(\epsilon_{i})_{i=1}^{\infty}\in\{0,1\}^{\mathbb{N}},$ is a \textit{universal expansion for $x$} if given any finite block $\delta_{1}\cdots\delta_{N}$ consisting of $0$'s and $1$'s, there exists $k\in\mathbb{N}$ such that $\epsilon_{k+i}=\delta_{i}$ for $1\leq i\leq N.$ Universal expansions were originally introduced in \cite{ErdosKomornik}, where they were shown to be intimately related to the so called spectra of a real number. We discuss this relation in more detail in Section $5.$ In \cite{Sid2} it was shown that for $\lambda \in (\frac{1}{2},1)$ almost every $x\in I_{\lambda}$ has a universal expansion. Proceeding by analogy with the case of $\lambda$-expansions, given $x\in\Lambda$ and $(\epsilon_{i})_{i=1}^{\infty}\in\{1,\ldots,n\}^{\mathbb{N}}$ a coding for $x.$ We say that $(\epsilon_{i})_{i=1}^{\infty}$ is a \textit{universal coding} if for any finite block $\delta_{1}\cdots\delta_{N}$ consisting of elements from $\{1,\ldots,n\},$ there exists $k\in \mathbb{N}$ such that $\epsilon_{k+i}=\delta_{i}$ for $1\leq i\leq N.$ Our result regarding universal codings is the following.

\begin{thm}
\label{Universal thm}
Suppose $\mathcal{L}(\Lambda)>0,$ then almost every $x\in\Lambda$ has a universal coding.
\end{thm}

The proofs of Theorems \ref{Uncountable thm} and \ref{Universal thm} will take on a similar structure. As such we will only prove Theorem \ref{Uncountable thm} in full and outline the necessary modifications required to prove Theorem \ref{Universal thm}.

The rest of this paper is structured as follows. In Section 2 we state some necessary preliminaries before giving our proofs of Theorems \ref{Uncountable thm}, \ref{Dimension thm} and \ref{Universal thm} in Section 3. In Section 4 we discuss some applications of our results to $\lambda$-expansions with arbitrary digit sets. Finally in Section 5 we pose some open questions.

\section{Preliminaries}
Before proving Theorems \ref{Uncountable thm}, \ref{Dimension thm} and \ref{Universal thm} we require the following technical arguments. For ease of exposition we denote the set of codings for a given $x\in \Lambda$ by $\Sigma_{\Lambda}(x),$ i.e., $$\Sigma_{\Lambda}(x):=\Big\{(\epsilon_{i})_{i=1}^{\infty}\in\{1,\ldots, n\}^{\mathbb{N}}: \lim_{N\to\infty} f_{\epsilon_{1}}\cdots f_{\epsilon_{N}} (0)=x\Big\}.$$Moreover, let $$U_{\Lambda}:=\Big\{x\in \Lambda: \textrm{card } \Sigma_{\Lambda}(x)=1\Big\}.$$ That is $U_{\Lambda}$ is the set of points with a unique coding.  Understanding the size/dimension of this set will be important in our proofs of Theorems \ref{Uncountable thm} and \ref{Dimension thm}.

Let $\{B_{k}\}_{k=1}^{\infty}$ be an enumeration of the set of all finite blocks consisting of elements from the set $\{1,\ldots,n\}.$ Moreover let $N_{k}$ denote the length of the block $B_{k}.$ To each $B_{k}$ we associate the set $U_{B_{k}}$ defined as follows: $$U_{B_{k}}:=\Big\{x\in \Lambda: \textrm{no coding of } x \textrm{ contains the block } B_{k} \Big\}.$$The following proposition highlights the importance of the set $U_{\Lambda}$ and the $U_{B_{k}}$'s.

\begin{prop}
\label{inclusion prop}
The following inclusions hold:
\begin{equation}
\label{equation 1}
U_{\Lambda}\subseteq\Big\{x\in \Lambda: \textrm{card }\Sigma_{\Lambda}(x)<2^{\aleph_{0}}\Big\},
\end{equation}
\begin{equation}
\label{equation 2}
U_{B_{k}}\subseteq \Big\{x\in \Lambda: x \textrm{ has no universal coding }\Big\} ,
\end{equation}
\begin{equation}
\label{equation 3}
 \Big\{x\in \Lambda: \textrm{card }\Sigma_{\Lambda}(x)<2^{\aleph_{0}}\Big\}\subseteq \bigcup_{N=0}^{\infty}\bigcup_{(\epsilon_{i})\in\{1,\ldots, n\}^{N}} f_{\epsilon_{1}}\cdots f_{\epsilon_{N}}(U_{\Lambda}),
\end{equation}
\begin{equation}
\label{equation 4}
\Big\{x\in \Lambda: x \textrm{ has no universal coding }\Big\}\subseteq \bigcup_{k=0}^{\infty}\bigcup_{N=0}^{\infty}\bigcup_{(\epsilon_{i})\in\{1,\ldots, n\}^{N}} f_{\epsilon_{1}}\cdots f_{\epsilon_{N}}(U_{B_{k}}).
\end{equation}
 \end{prop}
\begin{proof}
Statements (\ref{equation 1}) and (\ref{equation 2}) are obvious. The proof of (\ref{equation 3}) in the context of homogeneous contractions can be found in \cite{Sidorov4}, however their proof does not make use of the homogeneity of the contractions and easily translates over to the inhomogeneous case. As such we only show that (\ref{equation 4}) holds. Suppose $x\in \Lambda$ does not have a universal coding and let $\{B_{k}\}_{k=1}^{\infty}$ be as above. To begin with we ask whether $x$ has a coding containing the block $B_{1}.$ If it doesn't then $x\in U_{B_{1}}.$  Suppose otherwise and let $(\epsilon_{i})_{i=1}^{\infty}\in \Sigma_{\Lambda}(x)$ contain $B_{1}$. Let $M_{1}\in \mathbb{N}$ be such that $\epsilon_{M_{1}+1}\cdots \epsilon_{M_{1}+N_{1}}=B_{1}.$ Moreover, let $j_{1}\in\mathbb{N}$ be the unique natural number for which $B_{k}$ appears in $\epsilon_{1}\cdots \epsilon_{M_{1}+N_{1}}$ for every $1\leq k\leq j_{1},$ but $B_{j_{1}+1}$ does not appear in  $\epsilon_{1}\cdots \epsilon_{M_{1}+N_{1}}.$ Such a $j_{1}$ has to exist as $\epsilon_{1}\cdots \epsilon_{M_{1}+N_{1}}$ is of finite length. Now we consider all codings of $x$ that begin with $\epsilon_{1}\cdots \epsilon_{M_{1}+N_{1}}$ and ask whether one of these codings contains the block $B_{j_{1}+1}.$ If there doesn't exist such a coding then $\lim_{N\to\infty} f_{\epsilon_{M_{1}+N_{1}+1}}\cdots f_{\epsilon_{N}}(0)\in U_{B_{j_{1}+1}},$ which implies $x\in f_{1}\cdots f_{\epsilon_{M_{1}+N_{1}}} (U_{B_{j_{1}}+1})$. If there does exist such a coding we denote it by $(\epsilon_{i}^{2})_{i=1}^{\infty}$ and let $M_{2}\in\mathbb{N}$ be such that $\epsilon_{M_{2}+1}^{2}\cdots \epsilon_{M_{2}+N_{j_{1}+1}}^{2}=B_{j_{1}+1}.$ We then define $j_{2}\in\mathbb{N}$ to be the unique natural number such that $B_{k}$ appears in $\epsilon_{1}^{2}\cdots \epsilon_{M_{2}+N_{j_{1}+1}}^{2}$ for all $1\leq k\leq j_{2},$ but $B_{j_{2}+1}$ does not appear. We then ask whether there exists a coding for $x$ beginning with $\epsilon_{1}^{2}\cdots \epsilon_{M_{2}+N_{j_{1}+1}}^{2}$ that contains the block $B_{j_{2}+1}$. If such a coding doesn't exist we stop, if one does exist we repeat the above steps. Assuming the above process does not terminate then at the $n$-th iteration we have constructed a finite sequence containing the blocks $B_{1},\ldots, B_{n},$ and this sequence can be extended to an element of $\Sigma_{\Lambda}(x).$ If this process continues indefinitely then we will construct a universal coding for $x$. However, as $x$ has no universal coding this algorithm must at some point terminate. This yields  $K, M(K)\in\mathbb{N}$ and $(\epsilon_{i})_{i=1}^{\infty}\in \Sigma_{\Lambda}(x)$ such that $\lim_{N\to \infty} f_{\epsilon_{M_{K}+1}}\cdots f_{\epsilon_{N}}(0)\in U_{B_{K}}.$ In which case $x\in f_{\epsilon_{1}}\cdots f_{\epsilon_{M_{K}}}(U_{B_{K}})$ and we may deduce the inclusion in (\ref{equation 4}).
\end{proof}
The right hand side of (\ref{equation 4}) in Proposition \ref{inclusion prop} might seem excessive. We might naively expect that if $x\in \Lambda$ does not have a universal coding then $x\in U_{B_{k}},$ for some $k.$ However, even if $x$ has no universal coding we cannot discount the possibility that for each $B_{k}$ there exists $(\epsilon_{i}^{k})_{i=1}^{\infty}\in \Sigma_{\Lambda}(x)$ containing $B_{k}$.

The following corollary is an immediate consequence of Proposition \ref{inclusion prop} and the fact that our $f_{j}$'s are all similitudes.

\begin{cor}
\label{Simplifying corollary}
 The following statements hold:
\begin{itemize}
 \item $\mathcal{L}(\{x\in \Lambda: \textrm{card }\Sigma_{\Lambda}(x)<2^{\aleph_{0}}\})=0$ if and only if $\mathcal{L}(U_{\Lambda})=0.$
\item $\dim_{H}(\{x\in \Lambda: \textrm{card }\Sigma_{\Lambda}(x)<2^{\aleph_{0}}\})=\dim_{H}(U_{\Lambda}).$
\item $\mathcal{L}( \{x\in \Lambda: x \textrm{ has no universal coding }\})=0$ if and only if $\mathcal{L}(U_{B_{k}})=0$ for every $B_{k}.$

\end{itemize}

\end{cor}
By Corollary \ref{Simplifying corollary}, to show that Theorems \ref{Uncountable thm}, \ref{Dimension thm} and \ref{Universal thm} hold, it suffices to show that equivalent statements hold for $U_{\Lambda}$ and a typical $U_{B_{k}}.$

We now elaborate on the technical condition $\sum_{j=1}^{n}\lambda_{j}^{d}\neq 1$ stated in Theorem \ref{Uncountable thm}.

\begin{prop}
\label{intersect prop}
Assume $\mathcal{L}(\Lambda)>0$. If $\sum_{j=1}^{n}\lambda_{j}^{d}=1$ then $\mathcal{L}(f_{k}(\Lambda)\cap f_{l}(\Lambda))=0$ for all $1\leq k<l\leq n.$ However, if $\sum_{j=1}^{n}\lambda_{j}^{d}\neq 1$ there exists $1\leq k<l\leq n$ such that $\mathcal{L}(f_{k}(\Lambda)\cap f_{l}(\Lambda))>0.$
\end{prop}
\begin{proof}
It is a straighforward inductive argument to show that the following holds. Let $\{A_{j}\}_{j=1}^{n}$ be a finite collection of measurable sets with finite Lebesgue measure. Then 
\begin{equation}
\label{quasi inclusion}
\mathcal{L}(\cup_{j=1}^{n}A_{j})=\sum_{j=1}^{n}\mathcal{L}(A_{j})-\sum_{i=2}^{n} \mathcal{L}(\cup_{j=1}^{i-1}A_{j}\cap A_{i}).
\end{equation}
Let us assume $\sum_{j=1}^{n}\lambda_{j}^{d}=1$ and that there exists $1\leq k<l\leq n$ such that $\mathcal{L}(f_{k}(\Lambda)\cap f_{l}(\Lambda))>0.$ Without loss of generality we may assume that $k=1$ and $l=2$. We observe the following:
\begin{align*}
 \mathcal{L}(\Lambda)&=\mathcal{L}(\cup_{j=1}^{n}f_{j}(\Lambda))\\
&= \sum_{j=1}^{n}\mathcal{L}(f_{j}(\Lambda))-\sum_{i=2}^{n} \mathcal{L}(\cup_{j=1}^{i-1}f_{j}(\Lambda)\cap f_{i}(\Lambda))\\
&= \mathcal{L}(\Lambda)\sum_{j=1}^{n}\lambda_{j}^{d}-\sum_{i=2}^{n} \mathcal{L}(\cup_{j=1}^{i-1}f_{j}(\Lambda)\cap f_{i}(\Lambda)).
\end{align*}
In our second equality we have used equation (\ref{quasi inclusion}). It follows that $$0=\sum_{i=2}^{n} \mathcal{L}(\cup_{j=1}^{i-1}f_{j}(\Lambda)\cap f_{i}(\Lambda)).$$ However, this is not possible if $\mathcal{L}(f_{1}(\Lambda)\cap f_{2}(\Lambda))>0.$

Now let us assume that $\sum_{j=1}^{n}\lambda_{j}^{d}\neq 1$ and that $\mathcal{L}(\Lambda)>0.$ By the inclusion exclusion principle the following equation holds.
\begin{align*}
\mathcal{L}(\Lambda)&=\mathcal{L}(\cup_{j=1}^{n}f_{j}(\Lambda))\\
&= \sum_{j=1}^{n}\mathcal{L}(f_{j}(\Lambda))-\sum_{1\leq i<j\leq n }^{n}\mathcal{L}(f_{i}(\Lambda)\cap f_{j}(\Lambda))+\sum_{1\leq i<j<h\leq n}\mathcal{L}(f_{i}(\Lambda)\cap f_{j}(\Lambda)\cap f_{h}(\Lambda)) -\\
&\cdots +(-1)^{n-1}\mathcal{L}(\cap_{j=1}^{n}f_{j}(\Lambda)).
\end{align*}
Which by a simple manipulation implies 
\begin{align*}
\Big(\sum_{j=1}^{n}\lambda_{j}^{d}-1\Big)\mathcal{L}(\Lambda)&=\sum_{1\leq i<j\leq n }^{n}\mathcal{L}(f_{i}(\Lambda)\cap f_{j}(\Lambda))-\sum_{1\leq i<j<h\leq n}\mathcal{L}(f_{i}(\Lambda)\cap f_{j}(\Lambda)\cap f_{k}(\Lambda)) +\\
& \cdots +(-1)^{n}\mathcal{L}(\cap_{j=1}^{n}f_{j}(\Lambda)).
\end{align*} By our assumptions the left hand side of the above equation is not equal to zero. This implies the right hand side is also non zero and there must exist $1\leq k<l\leq n$ such that $\mathcal{L}(f_{k}(\Lambda)\cap f_{l}(\Lambda))>0.$
\end{proof}
We remark that if $x\in f_{k}(\Lambda)\cap f_{l}(\Lambda)$ then $x$ has at least two codings, one with first digit $k$ and one with first digit $l$. Moreover, it is straightforward to show that $x\in f_{\epsilon_{1}}\cdots f_{\epsilon_{N}}(f_{k}(\Lambda)\cap f_{l}(\Lambda))$ for some $(\epsilon_{i})_{i=1}^{N}\in\{1,\ldots,n\}^{N}$ and $1\leq k<l\leq n$ if and only if $x$ has at least two codings. This important remark will be used in the proof of the following corollary and later in our proof of Theorem \ref{Uncountable thm}.

\begin{cor}
\label{unique cor} 
Suppose $\sum_{j=1}^{n}\lambda_{j}^{d}=1$ and $\mathcal{L}(\Lambda)>0.$ Then almost every $x\in \Lambda$ has a unique coding.
\end{cor}
\begin{proof}
By the above remarks the following equality holds $$\Big\{x\in \Lambda:  \textrm{card }\Sigma_{\Lambda}(x)>1\Big\}= \bigcup_{1\leq k<l\leq n}\bigcup_{N=0}^{\infty}\bigcup_{(\epsilon_{i})\in\{1,\ldots, n\}^{N}} f_{\epsilon_{1}}\cdots f_{\epsilon_{N}}(f_{k}(\Lambda)\cap f_{l}(\Lambda)).$$ It is an immediate consequence of this equality, the fact that our $f_{j}$'s are all similitudes, and Proposition \ref{intersect prop} that $\mathcal{L}(\{x\in \Lambda:  \textrm{card }\Sigma_{\Lambda}(x)>1\})=0.$
\end{proof}

\section{Proof of Theorems \ref{Uncountable thm}, \ref{Dimension thm} and \ref{Universal thm}}
We begin by proving Theorems \ref{Uncountable thm} and \ref{Universal thm}. Their proofs will depend on an application of the Lebesgue density theorem. The Lebesgue density theorem states that if $E\subset \mathbb{R}^{d}$ is a Lebesgue measurable set, then for almost every $x\in E$ $$\lim_{r\to 0}\frac{\mathcal{L}(E\cap B_{r}(x))}{\mathcal{L}(B_{r}(x))}=1.$$ Here $B_{r}(x)$ denotes the closed $d$-dimensional ball in $\mathbb{R}^{d}$ with radius $r$ centred at $x.$ This statement is of course vacuous if $\mathcal{L}(E)=0.$ It is an immediate consequence of the Lebesgue density theorem that if $E\subset \mathbb{R}^{d}$ is such that every $x\in E$ satisfies $$\limsup_{r\to 0}\frac{\mathcal{L}(E\cap B_{r}(x))}{\mathcal{L}(B_{r}(x))}<1,$$ then $\mathcal{L}(E)=0.$ This will be the strategy will employ when it comes to proving Theorems \ref{Uncountable thm} and \ref{Universal thm}.

\begin{proof}[Proof of Theorem \ref{Uncountable thm}]
By Corollary \ref{Simplifying corollary} it suffices to show $\mathcal{L}(U_{\Lambda})=0.$ By Proposition \ref{intersect prop} we may assume that $1\leq k<l\leq n$ are such that $\mathcal{L}(f_{k}(\Lambda)\cap f_{l}(\Lambda))>0$. 

We now fix $x\in U_{\Lambda},$ and let $(\epsilon_{i})_{i=1}^{\infty}$ be its unique coding. Given $r>0$ we associate the unique $n(r)\in\mathbb{N}$ satisfying $$Diam(\Omega)\prod_{i=1}^{n(r)}\lambda_{\epsilon_{i}}<r\leq Diam(\Omega)\prod_{i=1}^{n(r)-1}\lambda_{\epsilon_{i}}.$$It is a consequence of these inequalities that $f_{\epsilon_{1}}\cdots f_{\epsilon_{n(r)}}(\Lambda)\subset B_{r}(x)$. 

We observe the following:
\begin{align*}
\frac{\mathcal{L}(U_{\Lambda}\cap B_{r}(x))}{\mathcal{L}(B_{r}(x))}&=1- \frac{\mathcal{L}(U_{\Lambda}^{c}\cap B_{r}(x))}{\mathcal{L}(B_{r}(x))}\\
&\leq 1- \frac{\mathcal{L}(U_{\Lambda}^{c}\cap f_{\epsilon_{1}}\cdots f_{\epsilon_{n(r)}}(\Lambda))}{\mathcal{L}(B_{r}(x))}\\
&\leq 1- \frac{\mathcal{L}(f_{\epsilon_{1}}\cdots f_{\epsilon_{n(r)}}(f_{k}(\Lambda)\cap f_{l}(\Lambda)))}{\mathcal{L}(B_{r}(x))}\\
&=1-\frac{\mathcal{L}(f_{k}(\Lambda)\cap f_{l}(\Lambda))\prod_{i=1}^{n(r)}\lambda_{\epsilon_{i}}^{d}}{C(d)r^{d}}\\
&\leq 1- \frac{\mathcal{L}(f_{k}(\Lambda)\cap f_{l}(\Lambda))\prod_{i=1}^{n(r)}\lambda_{\epsilon_{i}}^{d}}{C(d)(Diam(\Omega)\prod_{i=1}^{n(r)-1}\lambda_{\epsilon_{i}})^{d}}\\
&=1-\frac{\mathcal{L}(f_{k}(\Lambda)\cap f_{l}(\Lambda))\lambda_{n(r)}^{d}}{C(d)Diam(\Omega)^{d}}\\
&=1-\frac{\mathcal{L}(f_{k}(\Lambda)\cap f_{l}(\Lambda))\min_{1\leq j\leq n}\{\lambda_{j}^{d}\}}{C(d)Diam(\Omega)^{d}}.
\end{align*}
In the third line of the above we have used the fact that $f_{\epsilon_{1}}\cdots f_{\epsilon_{n(r)}}(f_{k}(\Lambda)\cap f_{l}(\Lambda))\subset f_{1}\cdots f_{\epsilon_{n(r)}}(\Lambda)$ and $f_{\epsilon_{1}}\cdots f_{\epsilon_{n(r)}}(f_{k}(\Lambda)\cap f_{l}(\Lambda))\subset U_{\Lambda}^{c}.$ Here $C(d)$ is the $d$-dimensional volume of the unit sphere. Clearly the upper density can therefore always be bounded above by some positive constant strictly less than $1$. Which by our earlier remarks implies $\mathcal{L}(U_{\Lambda})=0.$

\end{proof}
By Corollary \ref{Simplifying corollary} to prove Theorem \ref{Universal thm} it suffices to show $\mathcal{L}(U_{B_{k}})=0$ for each $B_{k}.$ This will follow from an analogous application of the Lebesgue density theorem. The role of $f_{k}(\Lambda)\cap f_{l}(\Lambda)$ is played by $f_{\epsilon_{1}}\cdots f_{\epsilon_{N_{k}}}(\Lambda)$ where $B_{k}= \epsilon_{1}\cdots \epsilon_{N_{k}}.$ Clearly $f_{\epsilon_{1}}\cdots f_{\epsilon_{N_{k}}}(\Lambda) \not\subset U_{B_{k}},$ it has measure $\mathcal{L}(\Lambda)\prod_{i=1}^{N_{k}}\lambda_{\epsilon_{i}}^{d}$ and its image under any finite sequence of $f_{j}$'s will also be in the complement of $U_{B_{k}}.$

We now prove Theorem \ref{Dimension thm}. The proof of this theorem is analogous to the proof of Theorem \ref{Nik's theorem} with one minor alteration. We begin by stating a lemma whose proof can be found in \cite{Sidorov4}.

\begin{lemma}
 \label{Nik lemma}
Let $A\subset \mathbb{R}^{d}$ be such that there exists a positive constant $\delta>0$ such that for an arbitrary cube $C\subset \mathbb{R}^{d}$ which intersects $A,$ one can find a cube $C_{0}\subset C$ such that $\mathcal{L}(C_{0})\geq \delta \mathcal{L}(C)$ and $C_{0}\cap A=\emptyset.$ Then $\dim_{H}(A)<d.$
\end{lemma}
The proof of Lemma \ref{Nik lemma} is fairly straighforward and follows from a box counting argument.

\begin{proof}[Proof of Theorem \ref{Dimension thm}]
 By Corollary \ref{Simplifying corollary} it suffices to show $\dim_{H}(U_{\Lambda})<d.$ We now show that Lemma \ref{Nik lemma} can be applied with $A=U_{\Lambda}.$ By our assumption $f_{k}(\Lambda)\cap f_{l}(\Lambda)$ has nonempty interior and therefore contains a $d$-dimensional cube that we shall denote by $C^{*}.$ We will show that we can take  $$\delta=\min\Big\{2^{-d},\frac{\mathcal{L}(C^{*})\min_{1\leq j\leq n}\{\lambda_{j}^{d}\}}{(2Diam(\Omega))^{d}}\Big\}.$$

Let $C(z,r)$ denote the cube in $\mathbb{R}^{d}$ centred at $z$ with side length $r.$ Suppose $C(z,r)$ intersects $U_{\Lambda}$. We ask whether $U_{\Lambda}$ intersects $C(z,\frac{r}{2}).$ If it doesn't we can take $C_{0}=C(z,\frac{r}{2})$ and $\mathcal{L}(C_{0})=2^{-d}\mathcal{L}(C(z,r)).$ Suppose otherwise, let $x\in U_{\Lambda}\cap C(z,\frac{r}{2})$ and $(\epsilon_{i})_{i=1}^{\infty}\in \Sigma_{\Lambda}(x).$ We let $n(r)\in\mathbb{N}$ denote the unique natural number satisfying the following inequalities $$Diam(\Omega)\prod_{i=1}^{n(r)}\lambda_{\epsilon_{i}}<\frac{r}{2}\leq Diam(\Omega)\prod_{i=1}^{n(r)-1}\lambda_{\epsilon_{i}}.$$ Clearly $f_{\epsilon_{1}}\cdots f_{\epsilon_{n(r)}}(\Lambda)\subset C(z,r)$ and therefore $f_{\epsilon_{1}}\cdots f_{\epsilon_{n(r)}}(C^{*})\subset C(z,r).$ Moreover $f_{\epsilon_{1}}\cdots f_{\epsilon_{n(r)}}(C^{*})$ is a cube and it is contained in $U_{\Lambda}^{c}.$ Finally we observe $$\mathcal{L}(f_{1}\cdots f_{\epsilon_{n(r)}}(C^{*}))= \mathcal{L}(C^{*})\prod_{i=1}^{n(r)}\lambda_{\epsilon_{i}}^{d}\geq \frac{r^{d}\mathcal{L}(C^{*})\lambda_{\epsilon_{n(r)}}^{d}}{(2Diam(\Omega))^{d}}\geq \frac{\mathcal{L}(C^{*})\min_{1\leq j\leq n}\{\lambda_{j}^{d}\}}{(2Diam(\Omega))^{d}}\mathcal{L}(C(z,r)).$$ Taking $C_{0}=f_{\epsilon_{1}}\cdots f_{\epsilon_{n(r)}}(C^{*})$ we see that our value for $\delta$ applies. Applying Lemma \ref{Nik lemma} yields our result.
\end{proof}

\section{Applications to $\lambda$-expansions with deleted digits}
Instead of considering $\lambda$-expansions where $\lambda\in(\frac{1}{2},1)$ and our sequences are elements of $\{0,1\}^{\mathbb{N}},$ we can consider the more general case where $\lambda\in(0,1)$ and the elements of our sequences are elements of $\mathcal{A}=\{a_{1},\ldots, a_{n}\}.$ Here $a_{j}\in \mathbb{R}$ for all $1\leq j \leq n,$ and without loss of generality we may assume that $a_{1}<\cdots<a_{n}.$ We refer to $\mathcal{A}$ as our \textit{alphabet}. Given $x\in[\frac{a_{1}\lambda}{1-\lambda},\frac{a_{n}\lambda}{1-\lambda}]$ we say that a sequence $(\epsilon_{i})_{i=1}^{\infty}\in \mathcal{A}^{\mathbb{N}}$ is a \textit{$\lambda$-expansion for $x$ with respect to $\mathcal{A}$} if $$x=\sum_{i=1}^{\infty}\epsilon_{i}\lambda^{i}.$$ We define the analogue of a universal expansion with respect to $\mathcal{A}$ in the natural way. Pedicini in \cite{Ped} showed that every $x\in [\frac{a_{1}\lambda}{1-\lambda},\frac{a_{n}\lambda}{1-\lambda}]$ has a $\lambda$-expansion with respect to $\mathcal{A}$ if and only if $$\max_{1\leq j\leq n-1}(a_{j+1}-a_{j})\leq \frac{\lambda(a_{m}-a_{1})}{1-\lambda}.$$ To the alphabet $\mathcal{A}$ we associate the set of maps $\{f_{j}\}_{j=1}^{n}$ where $f_{j}(x)=\lambda x +\lambda a _{j}.$ It is straightforward to show that $(\epsilon_{i})_{i=1}^{\infty}\in\{1,\ldots,n\}^{\mathbb{N}}$ is a coding for $x$ if and only if $(a_{\epsilon_{i}})_{i=1}^{\infty}\in \mathcal{A}^{\mathbb{N}}$ is a $\lambda$-expansion of $x$ with respect to the alphabet $\mathcal{A}.$ Therefore $\Lambda$ coincides with the set of points that have a $\lambda$-expansion with respect to this alphabet. As such, when the Pedicini condition is satisfied $\Lambda=[\frac{a_{1}\lambda}{1-\lambda},\frac{a_{n}\lambda}{1-\lambda}].$ In which case Theorem \ref{Universal thm} applies and we have the following result.

\begin{thm}
 Let $\mathcal{A}=\{a_{1},\ldots,a_{n}\}$ and suppose that $\lambda\in(0,1)$ is such that the Pedicini condition is satisfied. Then almost every $x\in [\frac{a_{1}\lambda}{1-\lambda},\frac{a_{n}\lambda}{1-\lambda}]$ has a universal expansion with respect to $\mathcal{A}$.
\end{thm}It was previously shown in \cite{Sidorov4} that when the Pedicini condition is satisfied and there exists $j$ for which $(a_{j+1}-a_{j})< \frac{\lambda(a_{m}-a_{1})}{1-\lambda},$ then almost every $x\in[\frac{a_{1}\lambda}{1-\lambda},\frac{a_{n}\lambda}{1-\lambda}]$ has a continuum of expansions. 

We now show that our results also translate over to the case of $\lambda$-expansions where the Pedicini condition is not satisfied. We now fix our alphabet to be $\mathcal{A}=\{0,1,3\}.$ Let $$I_{\lambda,\mathcal{A}}:=\Big\{x: x=\sum_{i=1}^{\infty}\epsilon_{i}\lambda^{i} \textrm{ for some } (\epsilon_{i})_{i=1}^{\infty}\in\mathcal{A}^{\mathbb{N}}\Big\}$$ The study of $\lambda$-expansions with respect to this alphabet and the set $I_{\lambda,\mathcal{A}}$ has received a lot of attention. We refer the reader to \cite{PolSim} and the references therein. In \cite{Sol} it was shown that for almost every $\lambda\in (\frac{1}{3},\frac{4}{5})$ the Lebesgue measure of $I_{\lambda,\mathcal{A}}$ is positive. Applying Theorems \ref{Uncountable thm} and \ref{Universal thm} we have the following result.

\begin{thm}
For almost every $\lambda\in(\frac{1}{3},\frac{4}{5})$ almost every $x\in I_{\lambda,\mathcal{A}}$ has a continuum of $\lambda$-expansions and a universal expansion.
\end{thm}We remark that for all $\lambda\in(\frac{1}{3},\frac{4}{5})$ the Pedicini condition is not satisfied. The above theorem therefore demonstrates cases where the Pedicini condition is not satisfied yet almost every $x\in \Lambda$ has a continuum of $\lambda$-expansions and a universal expansion.

\section{Open problems}

We conclude by posing some open questions and giving some general discussion.

\begin{itemize}
\item Let $\lambda\in(\frac{1}{2},1)$ and $$X(\lambda):=\Big\{\sum_{i=0}^{n}\epsilon_{i}\lambda^{-i}: \epsilon_{i}\in\{0,1\} \textrm{ and } n=0,1,\ldots \Big\}.$$ $X(\lambda)$ is a discrete set and may therefore be written as $\{y_{k}(\lambda)\}_{k=1}^{\infty}$ where $y_{1}(\lambda)<y_{2}(\lambda)<\ldots$. We introduce the following limits $$l(\lambda)=\liminf_{k\to \infty} y_{k+1}(\lambda)-y_{k}(\lambda) \textrm{ and }  L(\lambda)=\limsup_{k\to \infty} y_{k+1}(\lambda)-y_{k}(\lambda).$$ The set $X(\lambda)$ and the limits $l(\lambda)$ and $L(\lambda)$ have received a lot of attention. For more information on this topic we refer the reader to \cite{Erdos}, \cite{ErdosKomornik}, \cite{AkiKom} and the references therein. The classification of those $\lambda$ for which $l(\lambda)=0$ was completed in a recent paper by Feng, see \cite{Feng}. It was shown that $l(\lambda)=0$ if and only if $\lambda^{-1}$ is not a Pisot number. However, we are interested in a result stated in \cite{ErdosKomornik} which states that every $x\in(0,\frac{\lambda}{1-\lambda})$ has a universal expansion with respect to the alphabet $\{0,1\}$ if $L(\lambda)=0.$ Given this connection between the set $X(\lambda)$ and the existence of universal expansions the following question seems natural: For a general $\Lambda$ can we construct a set which is in some sense natural, and plays a similar role as $X(\lambda)$ does for $\lambda$-expansions? That is, does there exist $E\subset \mathbb{R}^{d}$ for which some sort of clustering property occuring within $E$ as we get further away from the origin implies the existence of universal codings for every point in $\textrm{int}(\Omega)\cap\Lambda.$ The author expects that such a set $E$ will exist. Our main motivation for posing this question is that we anticipate once we know how to define $E$ lots of other interesting question will arise. For example, once the analogues of $l(\lambda)$ and $L(\lambda)$ are established, when do they equal zero?
\item As stated earlier we can construct a self-similar set with positive Lebesgue measure when $\sum_{j=1}^{n}\lambda_{j}^{d}=1$. However, the example we gave was somewhat unsatisfactory. When $p_{1}=0,$ $p_{2}=1$ and $\lambda=\frac{1}{2}$ the images of $f_{1}([0,1])$ and $f_{2}([0,1])$ intersect in a trivial way. We would be very interested to know whether there exists an example of a self-similar set with positive Lebesgue measure when $\sum_{j=1}^{n}\lambda_{j}^{d}= 1$ for which the overlaps are nontrivial. More specifically, does there exist a self-similar set with positive Lebesgue measure when $\sum_{j=1}^{n}\lambda_{j}^{d}= 1$ for which there exists $1\leq k<l\leq n$ such that $f_{k}(\Omega)\cap f_{l}(\Omega)$ has nonempty interior. 
\item In the case of $\lambda$-expansions with respect to the alphabet $\{0,1\}$ what can be said about the Hausdorff dimension of the set of $x\in I_{\lambda}$ with no universal expansion. For $\lambda$ sufficiently close to one it can be shown that $L(\lambda)=0$ and the set of points that do not have a universal expansion are precisely the endpoints of $I_{\lambda}.$ However, we can assert that the Hausdorff dimension is positive when $\lambda\in (\frac{1}{2}, \lambda^{*}),$ where $\lambda^{*}$ is the Komornik Loreti constant. This is a straightforward consequence of the fact that $x\in I_{\lambda}$ with a unique $\lambda$-expansion cannot be universal, combined with the aforementioned results of \cite{GlenSid} which state that for $\lambda\in(\frac{1}{2}, \lambda^{*})$ the Hausdorff dimension of the set of $x\in I_{\lambda}$ with unique $\lambda$-expansion is positive. In particular, we would be interested in determining for which values of $\lambda\in(\frac{1}{2},1)$ is the Hausdorff dimension of the set of points with no universal expansion positive. 
\end{itemize}

\noindent \textbf{Acknowledgements} The author would like to thank Tom Kempton and Nikita Sidorov for useful discussions. This work was supported by the Dutch Organisation for Scientiﬁc Research (NWO) grant number
613.001.022.

\end{document}